\documentclass{amsart}
\pdfoutput=1
\usepackage{amssymb}
\usepackage{amsmath}
\usepackage{amsthm} 
\usepackage[colorlinks]{hyperref}

\usepackage[all,cmtip]{xy} 

\newtheorem{theorem}{Theorem}
\newtheorem{proposition}[theorem]{Proposition}
\newtheorem{question}[theorem]{Question}

\theoremstyle{remark}
\newtheorem{remark}[theorem]{Remark}

\DeclareMathOperator{\End}{End}

\newcommand{\BC}{{\mathbf{C}}}
\newcommand{\BQ}{{\mathbf{Q}}}
\newcommand{\BZ}{{\mathbf{Z}}}

\newcommand\lowtilde{\hbox{\lower0.7ex\hbox{\char`\~}}}

\renewcommand{\hat}{\widehat}

\begin{document}

\title[Surjections of Mordell--Weil groups]
{Optimal quotients and surjections of Mordell--Weil groups}

\author{Everett W. Howe}
\address{Center for Communications Research,
        4320 Westerra Court,
        San Diego, CA 92121-1967, USA.}
\email{however@alumni.caltech.edu}
\urladdr{\href{http://www.alumni.caltech.edu/~however/}
             {http://www.alumni.caltech.edu/\lowtilde{}however/}}

\date{6 February 2016}
\keywords{Jacobian, Mordell--Weil group, automorphism}

\subjclass[2010]{Primary 11G10; Secondary 11G05, 11G30, 11G35}

\begin{abstract}
Answering a question of Ed Schaefer, we show that if $J$ is the Jacobian of a
curve $C$ over a number field, if $s$ is an automorphism of $J$ coming from an 
automorphism of~$C$, and if $u$ lies in $\BZ[s]\subseteq\End J$ and has 
connected kernel, then it is not necessarily the case that $u$ gives a 
surjective map from the Mordell--Weil group of $J$ to the Mordell--Weil group
of its image.
\end{abstract}

\maketitle

\section{Introduction}
\label{S:intro}

Let $J$ be the Jacobian of a curve $C$ over a number field. If the automorphism
group $G$ of $J$ is nontrivial, one can use idempotents of the group algebra 
$\BQ[G]$ to decompose $J$ (up to isogeny) as a direct sum of abelian 
subvarieties.  This decomposition can be useful, for example, if one would like
to compute the rational points on~$C$, because one of the subvarieties may 
satisfy the conditions necessary for Chabauty's method even when $J$ itself
does not.

In this context, Ed Schaefer asked the following question in an online 
discussion:

\begin{question}
\label{Q:Ed}
Let $C$ be a curve over a number field~$k$, let $\sigma$ be a nontrivial
automorphism of~$C$, let $s$ be the associated automorphism of the Jacobian $J$
of~$C$, and let $u$ be an element of $\BZ[s]\subseteq \End J$. Let 
$A\subseteq J$ be the image of~$u$, and suppose the kernel of $u$ is connected.
Is it always true that map of Mordell--Weil groups $J(k)\to A(k)$ induced by
$u$ is surjective\textup{?}
\end{question}

An \emph{optimal quotient} of an abelian variety $A$ is a surjective morphism 
$A\to A'$ of abelian varieties whose kernel is connected
(see~\cite[\S{3}]{ConradStein2001}), so Schaefer's question asks whether an 
optimal quotient of a curve's Jacobian ``coming from'' an automorphism of the 
curve necessarily induces a surjection on Mordell--Weil groups.

The purpose of this paper is to show by explicit example that the answer to 
Schaefer's question is \emph{no}. In Section~\ref{S:double} we show that if 
$\varphi\colon C\to E$ is a degree-$2$ map from a genus-$2$ curve to an 
elliptic curve, and if $\sigma$ is the involution of $C$ that fixes~$E$, then
the endomorphism $1+s$ of $J$ has connected kernel and its image is isomorphic
to~$E$.  In fact, the map $J\to E$ determined by $1+s$ is isomorphic to the 
push-forward $\varphi_*\colon J\to E$.  To show that the answer to
Question~\ref{Q:Ed} is \emph{no}, it therefore suffices to find a double cover
$\varphi\colon C\to E$ of an elliptic curve by a genus-$2$ curve such that 
$\varphi_*$ is not surjective on Mordell--Weil groups.  We provide one such 
example in Section~\ref{S:example}, and show in Section~\ref{S:examples} that
there are in fact infinitely many examples.

\section{Genus-2 double covers of elliptic curves}
\label{S:double}

In this section we review some facts about genus-$2$ double covers of elliptic
curves over an arbitrary field of characteristic not~$2$.  In 
Section~\ref{S:example} we will return to the case where the base field is a 
number field.
 
The general theory of degree-$n$ maps from genus-$2$ curves to elliptic curves
is explained by Frey and Kani~\cite{FreyKani1991}. Over the complex numbers, 
the complete two-parameter family of genus-$2$ double covers of elliptic curves
was given in 1832 by Jacobi~(\cite[pp.~416--417]{Jacobi1832}, 
\cite[pp.~380--382]{Jacobi}) as a postscript to his review of Legendre's 
\emph{Trait\'e des fonctions elliptiques}~\cite{Legendre1828}; Legendre had 
himself given a one-parameter family of genus-$2$ double covers of elliptic 
curves (see Remark~\ref{R:Legendre}, below).  
In~\cite[\S{}3.2]{HoweLeprevostEtAl2000}, Jacobi's construction is modified so
that it works rationally over any base field of characteristic not~$2$, as
follows:

Let $k$ be an arbitrary field of characteristic not~$2$ and let $K$ be a
separable closure of~$k$. Suppose we are given equations $y^2 = f$ and
$y^2 = g$ for two elliptic curves $E$ and $F$ over~$k$, where $f$ and $g$ are
separable cubics in $k[x]$, and suppose further that we are given an
isomorphism $\psi\colon E[2](K)\to F[2](K)$ of Galois modules such that $\psi$
is not the restriction to $E[2]$ of a geometric isomorphism $E_K\to F_K$.  
Then~\cite[Proposition~4, p.~324]{HoweLeprevostEtAl2000} gives an explicit 
equation for a genus-$2$ curve $C/k$ such that the Jacobian $J$ of $C$ is 
isomorphic to the quotient of $E\times F$ by the graph $\Gamma$ of~$\psi$. (We 
say that $C$ is the curve obtained by \emph{gluing} $E$ and $F$ together along 
their $2$-torsion using~$\psi$.) Let $\omega$ be the quotient map from 
$E\times F$ to~$J$.  The construction from~\cite{HoweLeprevostEtAl2000} also 
shows that if $\lambda\colon J\to\hat{J}$ is the canonical principal 
polarization on~$J$, then there is a diagram
\begin{equation}
\label{eq-polarization}
\xymatrix{
E\times F\ar[rr]^{(2,2)}\ar[d]_{\omega} && E\times F\\ 
J\ar[rr]^\lambda                        && \hat{J}\ar[u]_{\hat{\omega}}.
}
\end{equation}
The automorphism $(1,-1)$ of $E\times F$ fixes $\Gamma$ and respects the product
polarization on $E\times F$, so it descends to give an automorphism $s$ of the 
polarized variety $(J,\lambda)$.  By Torelli's 
theorem~\cite[Theorem~12.1, p.~202]{Milne1986}, the automorphism $s$ comes from 
an automorphism $\sigma$ of~$C$.  Clearly $\sigma$ has order~$2$, and the 
quotient of $C$ by the order-$2$ group $\langle \sigma\rangle$ is isomorphic
to~$E$.  Let $\varphi\colon C\to E$ be the associated double cover.

Let $u = 1 + s\in \End J$.  Then we have a diagram
\begin{equation}
\label{eq-diagram}
\xymatrix{
E\times F\ar[rr]^{(2,0)}\ar[d]_{\omega} && E\times F\ar[d]^\omega\\ 
J\ar[rr]^{u}                            && J.
}
\end{equation}
We claim that the kernel of $u$ is connected. To see this, note that the kernel 
of $\omega\circ (2,0)$ is simply $E[2]\times F$. The image of $E[2]\times F$ in
$J$ (under the map~$\omega$) is equal to the image of $0\times F$ in $J$ because
every element of $E[2]\times 0$ is congruent modulo $\Gamma$ to an element of 
$0\times F[2]$. Also, since $\Gamma$ intersects $0\times F$ only in the 
identity, the image of $F$ in $J$ is isomorphic to~$F$, so the kernel of $u$ is 
isomorphic to~$F$.

On the other hand, we see from diagram~\eqref{eq-diagram} that the \emph{image}
of $u$ is equal to the image of $E\times 0$ in~$J$.  Since $\Gamma$ has trivial 
intersection with $E\times 0$, the image of $u$ is isomorphic to~$E$.  The 
induced map $J\to E$ is nothing other than $\varphi_*$.

Likewise, the involution $-s$ on $J$ corresponds to an involution $\sigma'$ 
of~$C$.  The quotient of $C$ by the group $\langle \sigma'\rangle$ is 
isomorphic to~$F$, and gives us a double cover $\varphi'\colon C\to F$.
If we set $v = 1 - s$, then $v$ has kernel isomorphic to $E$ and image 
isomorphic to $F$, and the map $J\to F$ induced by $v$ is $\varphi'_*$.

\begin{remark}
Frey and Kani prove a more general result: Given two elliptic curves $E$ and
$F$ over an algebraically closed field~$k$, an integer $n>1$, and an 
isomorphism $\psi\colon E[n]\to F[n]$ of group schemes that is an anti-isometry
with respect to the Weil pairings on $E[n]$ and $F[n]$, there is a 
possibly-singular curve $C$ over $k$ of arithmetic genus $2$ whose polarized 
Jacobian $(J,\lambda)$ fits into a diagram analogous to \eqref{eq-polarization},
but with the $2$'s on the top arrow replaced with~$n$'s.  The curve $C$ has 
degree-$n$ maps to both $E$ and~$F$, and arguments like the one given above 
show that the corresponding push-forward maps from $J$ to $E$ and from $J$ to 
$F$ are optimal.
\end{remark}

\begin{remark}
\label{R:Legendre}
Legendre's family of genus-$2$ curves with split Jacobians
\cite[Troi\-si\`eme Suppl\'ement, \S{}XII, pp.~333--359]{Legendre1828} is the 
family over $\BC$ obtained from the construction above by taking $F = E$
and by taking $\psi\colon E[2](\BC)\to E[2](\BC)$ so that it fixes one point of
order $2$ and swaps the other two.
\end{remark}

In Sections~\ref{S:example} and~\ref{S:examples}, we will use the construction
that we have just described to produce genus-$2$ curves with involutions that
we can use to show that the answer to Question~\ref{Q:Ed} is \emph{no}. As part 
of our analyses, we will need to know how to tell whether a point of
$(E\times F)(k)$ lies in the image of $J(k)$ under the map 
$(\varphi_*,\varphi'_*)\colon J\to E \times F$.  Such a criterion is given in
Proposition~12 (p.~338) of~\cite{HoweLeprevostEtAl2000}.  For the reader's
convenience, we review that criterion here.  We continue to use the notation
set earlier in the section: $E$ and $F$ are elliptic curves given by
equations $y^2 = f$ and $y^2 = g$, respectively; 
$\psi\colon E[2](K) \to F[2](K)$ is an isomorphism of Galois modules;
and $C$ is a genus-$2$ curve whose Jacobian $J$ is isomorphic the quotient of
$E\times F$ by the graph of $\psi$.  The curve $C$ comes provided with covering
maps $\varphi\colon C\to E$ and $\varphi'\colon C\to F$ of degree~$2$, and the 
quotient map $E\times F\to J$ followed by $(\varphi_*,\varphi'_*)$ is 
multiplication-by-$2$ on $E\times F$.

Let $L$ be the $3$-dimensional $k$-algebra $k[x]/(f)$ and let $X$ be the image
of $x$ in~$L$. Note that $L$ is a product of fields, one for each Galois orbit
of $2$-torsion points in~$E(K)$.  The norm from $L$ to $k$ induces a map from
$L^*/L^{*2}$ to $k^*/k^{*2}$ that we continue to call the norm, and we let 
$\hat{L}$ be the kernel of the norm $L^*/L^{*2}\to k^*/k^{*2}$.  There is a
homomorphism $\iota\colon E(k)/2E(k) \to \hat{L}$ defined as follows: If 
$P\in E(k)$ is a rational non-$2$-torsion point with $x$-coordinate~$x_P$, then
$\iota$ sends the class of $P$ modulo $2E(k)$ to the class of $x_P - X$
modulo $L^{*2}$.  If $P\in E(k)$ is a rational point of order~$2$, then 
$x_P - X$ is nonzero in each component of $L$ other than the one corresponding 
to $P$; the value of $\iota$ on the class of $P$ is then the unique element of
$\hat{L}$ that agrees with $x_P - X$ on the components where it is nonzero.

Similarly, we define a $k$-algebra $L' = k[x]/(g)$ and a homomorphism 
$\iota'$ from $E(k)/2E(k)$ to $\hat{L}'$.  We note that the map $\psi$ induces 
an isomorphism $\psi^*: \hat{L}' \to \hat{L}$.

\begin{proposition}
\label{P:2torsion}
A point $(P,Q)\in (E\times F)(k)$ lies in the image of $J(k)$ under the map 
$(\varphi_*,\varphi'_*)$ if and only if the isomorphism $\psi^*$ takes
$\iota'(Q)$ to $\iota(P)$.
\end{proposition}

\begin{proof}
This follows immediately 
from~\cite[Proposition~12, p.~338]{HoweLeprevostEtAl2000}.
\end{proof}

\section{A small example}
\label{S:example}

Let $E$ and $F$ be the elliptic curves over $\BQ$ defined by $y^2 = f$ and 
$y^2 = g$, respectively, where
\[
f = x^3 + 5 x^2 + 6 x + 1 
\text{\qquad and\qquad}
g = x^3 - 6 x^2 + 5 x - 1.
\]
Let $K$ be the number field defined by the irreducible polynomial~$f$. 
Let $r$ be a root of $f$ in $K$; then $-r^2 - 4 r - 4$ and $r^2 + 3 r - 1$ 
are also roots of~$f$. Set 
\[
\alpha_1 = r, \qquad 
\alpha_2 = -r^2 - 4 r - 4, \qquad 
\alpha_3 = r^2 + 3 r - 1,
\]
and note that if we set $\beta_i = -1/\alpha_i$ then the $\beta$'s are the 
three roots of~$g$.

Let $\psi\colon E[2](K)\to F[2](K)$ be the isomorphism that sends 
$(\alpha_i,0)$ to $(\beta_i,0)$, for $i=1,2,3$. Using the formulas 
from~\cite[Proposition~4, p.~324]{HoweLeprevostEtAl2000}, we see that the
curve $C$ over $\BQ$ defined by $y^2 = 7^8 g(x^2)$ has Jacobian $J$ isomorphic 
to the quotient of $E\times F$ by the graph of~$\psi$. Rescaling~$y$, we find
that $C$ has a model 
\[
y^2 = x^6 - 6 x^4 + 5 x^2 - 1.
\]
The double cover $\varphi\colon C\to E$ is given by
$(x,y)\mapsto (-1/x^2, y/x^3)$, and the double cover $\varphi'\colon C\to F$ 
by $(x,y) \mapsto (x^2,y)$.

The curve $E$ is isomorphic to the curve 196A1 from Cremona's database; its 
Mordell--Weil group is generated by the point $P = (-2,1)$ of infinite order.
The curve $F$ is isomorphic to the curve 784F1 from Cremona's database, and its
Mordell--Weil group is trivial.

Let $\sigma$ be the involution $(x,y)\mapsto(-x,-y)$ of~$C$, so that $\sigma$ 
generates the Galois group of the cover $C\to E$, and let $s$ be the 
corresponding involution of~$J$.  We know from Section~\ref{S:double} that the
endomorphism $u = 1 + s$ of $J$ has connected kernel and has image isomorphic
to~$E$, and that the associated optimal cover $J\to E$ is simply~$\varphi_*$.
We claim that the point $P$ is not in the image under $\varphi_*$ of the 
Mordell--Weil group of~$J$.

We prove this claim by contradiction.  Suppose there were a point $R$ of 
$J(\BQ)$ with $\varphi_*(R) = P$. The only possible image for $R$ in $F(\BQ)$
is the identity element~$O$, so we must have 
$(\varphi_*,\varphi'_*)(R) = (P,O)$. Now we apply Proposition~\ref{P:2torsion}.
We see that the $\BQ$-algebra $L$ from the proposition is simply the field~$K$,
the group $\hat{L}$ is the quotient of the subgroup of elements of $K^*$ with 
square norm by the subgroup $K^{*2}$, and the map
$\iota\colon E(\BQ)/2E(\BQ) \to \hat{L}$ sends the class of a nonzero point
$(x,y)\in E(\BQ)$ to the class in $\hat{L}$ of the element $x-r\in K^*$.  
(Note that $x-r$ does lie in the subgroup of $K^*$ of elements whose norms are
squares, because the norm of $x-r$ is equal to $y^2$.)

Since $(P,O)$ lies in the image of $J(\BQ)$, Proposition~\ref{P:2torsion} says
$\iota(P)$ must be the trivial element of $\hat{L}$; that is, $-2-r$ must be a 
square in~$K$.  But $-2-r$ is \emph{not} a square in $K$; this can be seen, for 
example, by looking modulo~$13$. Therefore $P$ is not in the image of under 
$\varphi_*$ of the Mordell--Weil group of~$J$.

\section{Infinitely many examples}
\label{S:examples}

The specific example given in Section~\ref{S:example} was chosen because the
equations for the curves and the maps worked out to have small integer 
coefficients.  In this section we present a method for producing infinitely 
many examples, without concerning ourselves about the simplicity of the
equations we obtain.

Let $E$ and $F$ be two elliptic curves over $\BQ$ defined by equations 
$y^2 = f$ and $y^2 = g$, respectively, where $f$ and $g$ are monic cubic
polynomials in $\BQ[x]$ that split completely over~$\BQ$. Let $P_1$, $P_2$, 
$P_3$ be the points of order $2$ in $E(\BQ)$ and let $Q_1$, $Q_2$, $Q_3$ be 
the points of order $2$ in $F(\BQ)$.  Let $C$ be the genus-$2$ curve over $\BQ$ 
produced by gluing $E$ and $F$ together along their $2$-torsion subgroups using
the isomorphism $\psi\colon E[2](K)\to F[2](K)$ that takes $P_i$ to $Q_i$, for 
$i=1,2,3$. Let $\varphi\colon C\to E$ and $\varphi'\colon C\to F$ be the 
degree-$2$ maps associated to this data and let $J$ be the Jacobian of~$C$.  
Suppose $P$ is a rational point on~$E$.   We know that $P$ is in the image of 
$J(\BQ)$ under $\varphi_*$ if and only if there is a point $Q$ of $F(\BQ)$ such
that $(P,Q)$ is in the image of $J(\BQ)$ under the map $(\varphi_*,\varphi'_*)$.

Again Proposition~\ref{P:2torsion} tells us whether such a $Q$ exists. In this 
case, because the $2$-torsion points of $E$ and $F$ are all rational, the 
answer takes a slightly different shape than it did in the preceding section.
Let $Z$ be the subgroup of $(\BQ^*/\BQ^{*2})^3$ consisting of those triples 
$(r,s,t)$ whose product is equal to the trivial element of~$\BQ^*/\BQ^{*2}$.  
Then the group $\hat{L}$ from Proposition~\ref{P:2torsion} is isomorphic to~$Z$,
and the isomorphism can be chosen so that the homomorphism $\iota$ sends a 
non-$2$-torsion point $P$ of $E(\BQ)$ to the class in $Z$ of the triple
\[
(x(P) - x(P_1), x(P) - x(P_2), x(P) - x(P_3)).
\]
Likewise, $\hat{L}'$ is isomorphic to $Z$, and the isomorphism can be chosen
so that the homomorphism $\iota'$ sends a non-$2$-torsion point $Q$ of $F(\BQ)$
to the class of 
\[
(x(Q) - x(Q_1), x(Q) - x(Q_2), x(Q) - x(Q_3)).
\]
Under these identifications, the isomorphism $\psi^*$ is nothing other than
the identity on~$Z$.  Thus, Proposition~\ref{P:2torsion} says that a point 
$(P,Q)$ in $(E\times F)(\BQ)$ is in the image of $(\varphi_*,\varphi'_*)$ if
and only if $\iota(P) = \iota'(Q)$.

Suppose we are given an arbitrary elliptic curve $F/\BQ$ with rational points
$Q_1$, $Q_2$, $Q_3$ of order~$2$. We will show that there are infinitely many
geometrically distinct choices for $E/\BQ$ with rational points $P_1$, $P_2$, 
$P_3$ of order~$2$ such that if $\varphi\colon C\to E$ is constructed as 
above, then there is a point of infinite order in $E(\BQ)$ that is not 
contained in the subgroup of $E(\BQ)$ generated by the torsion elements and the
image of $J(\BQ)$ under~$\varphi_*$.

If $z$ is an element of $(\BQ^*/\BQ^{*2})^3$, we say that a prime $p$ 
\emph{occurs in $z$} if one of the components of $z$ has odd valuation at~$p$.
If $E$ is an elliptic curve over $\BQ$ with all of its $2$-torsion rational 
over~$\BQ$, we say that a prime $p$ \emph{occurs in $E(\BQ)$} if it occurs in
some element of~$\iota(E(\BQ))$; note that only finitely many primes occur in 
$E(\BQ)$ because $E(\BQ)$ is a finitely-generated group. Let $\ell_1$ and 
$\ell_2$ be two distinct odd primes that do not occur in $F(\BQ)$. Let $p$ be 
one of the infinitely many odd primes that do not occur in $F(\BQ)$ and that are
congruent to $\ell_1+1$ modulo $\ell_1^2$ and to $\ell_2-1$ modulo~$\ell_2^2$,
and let $E_p$ be the elliptic curve
\[
y^2 = x (x + p+1) (x - p + 1).
\]
Let $P_1$, $P_2$, and $P_3$ be the $2$-torsion points on $E_p$ with 
$x$-coordinates~$0$, $-p-1$, and $p-1$, respectively, and let 
$P = (-1,p)\in E_p(\BQ)$.  We compute that the images of these points in 
$Z \subset (\BQ^*/\BQ^{*2})^3$ are as follows:
\begin{align*}
\iota(P)   &= (-1, p, -p) \\
\iota(P_1) &= (-p^2+1, p+1, -p+1) \\
\iota(P_2) &= (-p-1, 2p(p+1),-2p) \\
\iota(P_3) &= (p-1,2p,2p(p-1)).
\end{align*}
We see that $p$ occurs in $\iota(P)$, that $p$ occurs in $\iota(P+P_1)$, that
$\ell_2$ occurs in $P+P_2$, and that $\ell_1$ occurs in $P+P_3$.

Note that $\iota(P_1)$, $\iota(P_2)$, and $\iota(P_3)$ are nontrivial, because
either $\ell_1$ or $\ell_2$ occurs in each of them. This shows that none of the
points $P_1$, $P_2$, and $P_3$ is the double of a rational point.  Since we
know the possible torsion structures of elliptic curves 
over~$\BQ$~\cite[Theorem~8, p.~35]{Mazur1977}, we see that $E_p$ has torsion
subgroup isomorphic to either $(\BZ/2\BZ)\times (\BZ/2\BZ)$ or
$(\BZ/2\BZ)\times (\BZ/6\BZ)$.  If there is a rational $3$-torsion point $T$
on~$E$, then $\iota(T) = (1,1,1)$, because $T$ is twice~$-T$. Combining this
with what we have already shown, we find that $\iota(P)$ is not contained in
the group generated by $\iota'(F(\BQ))$ and the image under $\iota$ of the
torsion subgroup of $E(\BQ)$.  From this, we see that $P$ is not contained in
the subgroup of $E(\BQ)$ generated by the torsion elements and the image of
$J(\BQ)$ under~$\varphi_*$.

Finally, we note that the $j$-invariant of $E_p$ is given by
\[
j(E_p) = \frac{64 (3 p + 1)^3}{p^2 (p-1)^2 (p+1)^2},
\]
so that, since $p$ is odd, it is the largest prime for which $j(E_p)$ has
negative valuation.  Therefore distinct odd primes $p$ and $q$ give 
geometrically nonisomorphic curves $E_p$ and $E_q$, so there are infinitely 
many curves $E_p$ that we can glue to $F$ as above to get examples showing that
the answer to Schaefer's question is \emph{no}.

\bibliographystyle{hplaindoi} 
\bibliography{optimalMW}

\end{document}